\documentclass[11pt]{amsart}

\usepackage{amssymb}
\usepackage{mathrsfs}
\usepackage{amsfonts}
\usepackage{latexsym,amsmath,amsthm,amssymb,amsfonts}
\usepackage[usenames]{color}
\usepackage{xspace,colortbl}
\usepackage{graphicx}
\usepackage{tipa}

\newcommand{\be}{\begin{equation}}
\newcommand{\ee}{\end{equation}}
\newcommand{\beq}{\begin{eqnarray}}
\newcommand{\eeq}{\end{eqnarray}}

\newtheorem{prop}{Proposition}[section]

\newtheorem{remark}[prop]{Remark}

\def\begeq{\begin{equation}}
\def\endeq{\end{equation}}

\def\odot{\setbox0=\hbox{$\bigcirc$}\relax \mathbin {\hbox
to0pt{\raise.5pt\hbox to\wd0{\hfil $\wedge$\hfil}\hss}\box0 }}

\numberwithin{equation} {section}

\numberwithin{equation}{section}
\newtheorem{theorem}{\bf Theorem}[section]

\newtheorem{lemma}[theorem]{\bf Lemma}

\allowdisplaybreaks

\begin{document}
\title[Pogorelov type estimates]
 {Pogorelov type estimates for a class of Hessian quotient equations in Lorentz-Minkowski space $\mathbb{R}^{n+1}_{1}$}

\author{
 Chenyang Liu,~~Jing Mao$^{\ast}$,~~Yating Zhao}

\address{
Faculty of Mathematics and Statistics, Key Laboratory of Applied
Mathematics of Hubei Province, Hubei University, Wuhan 430062, China
}

\email{ 1109452431@qq.com, jiner120@163.com, 2505207483@qq.com}

\thanks{$\ast$ Corresponding author}

\date{}
\maketitle
\begin{abstract}
Let $\Omega$ be a bounded domain (with smooth boundary) on the
hyperbolic plane $\mathscr{H}^{n}(1)$, of center at origin and
radius $1$, in the $(n+1)$-dimensional Lorentz-Minkowski space
$\mathbb{R}^{n+1}_{1}$. In this paper, by using a priori estimates,
we can establish Pogorelov type estimates of $k$-convex solutions to
a class of Hessian quotient equations defined over
$\Omega\subset\mathscr{H}^{n}(1)$ and with the vanishing Dirichlet
boundary condition.
\end{abstract}

\maketitle {\it \small{{\bf Keywords}: Hessian quotient equations,
$k$-convex, Lorentz-Minkowski space, Dirichlet boundary condition, a
priori estimates.
 }

}

\section{Introduction}

As we know, Pogorelov \emph{firstly} \cite{po1} obtained Pogorelov's
$C^2$ interior estimates for Monge-Amp\`{e}re equations (see also
\cite{gt1}). Later, Chou and Wang \cite{cw1,w1} improved this result
to a more general setting -- they can obtain Pogorelov type
estimates for $k$-Hessian equations. More precisely, for a bounded
Euclidean domain $\Omega\subset\mathbb{R}^n$ with smooth boundary
$\partial\Omega$, they considered the following Dirichlet problem of
the $k$-Hessian equation
\begin{equation}\label{sk-1}
\left\{
\begin{aligned}
& \sigma_{k}(\nabla^{2}u)=f(x,u) \qquad &&\mathrm{in}~ \Omega,
\\
&u=0 \qquad&&\mathrm{on}~\partial\Omega,
\end{aligned}
\right.
\end{equation}
with as usual\footnote{~We make an agreement that in this paper,
$\sigma_{k}(\lambda(\cdot))$ would denote the $k$-th elementary
symmetric
 function of eigenvalues of a given tensor.}
\begin{eqnarray} \label{sk-def}
\sigma_{k}(\nabla^{2}u):=\sigma_{k}(\lambda(\nabla^{2}u))=\sigma_{k}(\lambda_{1},
\lambda_{2}, \cdots, \lambda_{n})=\sum\limits_{1\leq
i_{1}<i_{2}<\cdots<i_{k}\leq
n}\lambda_{i_{1}}\lambda_{i_{2}}\cdots\lambda_{i_{k}}
\end{eqnarray}
 the $k$-th elementary symmetric function of eigenvalues
$(\lambda_{1}, \lambda_{2}, \cdots,
\lambda_{n})=\lambda(\nabla^{2}u)$ of the Hessian matrix
$\nabla^{2}u$, and successfully proved that if $u\in C^{2}(\Omega)$
is a $k$-convex solution to the Hessian equation (\ref{sk-1}), then
there exists a constant $\beta>0$ such that
\begin{eqnarray*}
\sup\limits_{\Omega}(-u)^{\beta}|\nabla^{2}u|\leq C
\end{eqnarray*}
for some constant\footnote{~In the sequel, many constants would
appear, and, by abuse of notations and without any potential
confusion, we prefer to use the notation $C$ to represent some of
them.} $C$. Here, $u$ is said to be
 $k$-convex if $\lambda(\nabla^{2}u)$ belongs to the Garding's cone
 (see also (\ref{gard-co}))
\begin{eqnarray*}
\Gamma_{k}=\{\lambda\in\mathbb{R}^{n}|\sigma_{j}(\lambda)>0,
~~j=1,2,\ldots,k\}.
\end{eqnarray*}
In some literatures, ``$k$-convex" is also called ``$k$-admissible",
and we have also used this terminology -- see, e.g.,
\cite{GLM,ggm,glm}. Several years ago, by imposing a stronger
assumption, Li-Ren-Wang \cite{lrw} improved Chou-Wang's above result
to a more general setting, that is, if furthermore  $f$ in the RHS
of the first equation in (\ref{sk-1}) depends also on the gradient
term $\nabla u$, then they can also establish Pogorelov type
estimates for $(k+1)$-convex solutions to the Dirichlet problem of
the $k$-Hessian equation. Besides, there is a special case -- when
$k=2$, their $3$-convex assumption can be replaced by $2$-convexity
to derive the Pogorelov type estimates.
 Based on these conclusions, one might
ask a natural question as follows:
 \vspace{2mm}
\begin{itemize}
\item \textbf{Problem 1}. Whether Pogorelov type estimates are still valid for Hessian quotient equations
\begin{eqnarray*}
\frac{\sigma_{k}(\lambda(\nabla^{2}u))}{\sigma_{l}(\lambda(\nabla^{2}u))}=f(x,u,\nabla
u)
\end{eqnarray*}
or not?
\end{itemize}
\vspace{3mm} As far as we know, \textbf{Problem 1} is still open.

In 2020, Chu and Jiao \cite{cj} considered the Hessian type equation
with vanishing Dirichlet boundary condition (DBC for short)
\begin{equation}\label{hte-1}
\left\{
\begin{aligned}
& \sigma_{k}(\lambda(\widetilde{U}[u]))=f(x,u,\nabla u) \qquad
&&\mathrm{in}~ \Omega,
\\
&u=0 \qquad&&\mathrm{on}~\partial\Omega
\end{aligned}
\right.
\end{equation}
on a bounded domain $\Omega\subset\mathbb{R}^n$ with smooth boundary
$\partial\Omega$, where $\widetilde{U}[u]:=(\Delta
u)\cdot\mathbb{I}-\nabla^{2}u$ with $\Delta u$ the Laplacian of $u$
and $\mathbb{I}$ the identity matrix, and, as before,
$\sigma_{k}(\lambda(\widetilde{U}[u]))$ stands for the $k$-th
elementary symmetric function of eigenvalues of the $(0,2)$-type
tensor $\widetilde{U}[u]$. They can establish Pogorelov type
estimates for $k$-convex solutions to (\ref{hte-1}). Recently,
inspired by this work, Chen, Tu and Xiang \cite{ctx} investigated
the Hessian quotient version of (\ref{hte-1}) as follows
\begin{equation}\label{hte-2}
\left\{
\begin{aligned}
&
\frac{\sigma_{k}(\lambda(\widetilde{U}[u]))}{\sigma_{l}(\lambda(\widetilde{U}[u]))}=f(x,u,\nabla
u) \qquad &&\mathrm{in}~ \Omega\subset\mathbb{R}^n, \qquad l+2\leq
k\leq n,~l>0,
\\
&u=0 \qquad&&\mathrm{on}~\partial\Omega,
\end{aligned}
\right.
\end{equation}
and successfully obtained Pogorelov type estimates for $k$-convex
solutions to (\ref{hte-2}), which improves Chu-Jiao's result
mentioned above a lot. When  $\widetilde{U}[u]$ in the problem
(\ref{hte-2}) was replaced by $\tau (\Delta
u)\cdot\mathbb{I}-\nabla^2 u$, $\tau\geq1$, Qing Dai have
established the similar result in her master's dissertation
\cite{qd}.

Caffarelli-Nirenberg-Spruck \cite{cns1,cns2} and Trudinger \cite{nt}
considered a class of fully nonlinear elliptic equations with DBC,
which covers (\ref{hte-2}) as a special case, and obtained the
existence of solutions under suitable assumptions. It is easy to see
that if $k=n$ in (\ref{hte-1}), then the corresponding Dirichlet
problem degenerates into a Monge-Amp\`{e}re type equation with
vanishing DBC. Harvey-Lawson \cite[Example 4.3.2]{hl4} studied the
$p$-convex ($1\leq p\leq n$) solutions to this Monge-Amp\`{e}re type
equation with vanishing DBC and solved the Dirichlet problem with
$f=0$ on suitable domains. Besides, they introduced the
$(n-1)$-convexity\footnote{~A function $u$ is called $(n-1)$-convex
if the matrix
 $
 (\Delta u)\cdot\mathbb{I}-\nabla^{2}u
 $
 is nonnegative definite. } or the general $p$-convexity ($1\leq p\leq n$) for
solutions to the Dirichlet problem of nonlinear elliptic equations
in a series of papers \cite{hl1,hl2,hl3}.

Similar to the concept of $(n-1)$-convexity, for a
\emph{complex-valued} function $u$, it is called
$(n-1)$-plurisubharmonic if $
 (\Delta u)\cdot\mathbb{I}-\nabla^{2}u
 $ is nonnegative definite, and this concept was also introduced by
 Harvey-Lawson \cite{hl2-1,hl2}. Clearly, complex Monge-Amp\`{e}re type
 equations for $(n-1)$-plurisubharmonic functions can be defined as
 follows
  \begin{eqnarray} \label{cma}
 \det\left(\left(\sum\limits_{m=1}^{n}\frac{\partial^{2}u}{\partial
 z_{m}\partial\bar{z}_{m}}\right)\delta_{ij}-\frac{\partial^{2}u}{\partial
 z_{i}\partial\bar{z}_{j}}\right)=f.
  \end{eqnarray}
For strict pseudo-convex domains in the complex Euclidean $n$-space
$\mathbb{C}^{n}$, the Dirichlet problem of (\ref{cma}) with positive
$f$ was solved by Li \cite{lsy}. On compact complex manifolds, the
Eq. (\ref{cma}) was investigated by Tosatti-Weinkove
\cite{tow1,tow2}, and moreover, it was shown that this equation has
relation with the Gauduchon conjecture in complex geometry (see
\cite{gp,stw}). For more progresses on the Eq. (\ref{cma}), we refer
readers to \cite{fww1,fww2,gs1} and references therein.

Elliptic equations of type similar to the one in (\ref{hte-1}) also
arise naturally in conformal geometry --  see \cite[pp.
274-275]{ctx} for a brief explanation and check \cite{sz} for
details.

Therefore, from the above introduction, it should be interesting and
meaningful to study the Dirichlet problem (\ref{hte-1}) and its more
general version (\ref{hte-2}).

In this paper, we study the Lorentz-Minkowski version of the problem
(\ref{hte-2}) and try to get the related Pogorelov type estimates.
In order to state our main result clearly, we need to give some
notions. Throughout this paper, let $\mathbb{R}^{n+1}_{1}$ be the
$(n+1)$-dimensional ($n\geq2$) Lorentz-Minkowski space with the
following Lorentzian metric
\begin{eqnarray*}
\langle\cdot,\cdot\rangle_{L}=dx_{1}^{2}+dx_{2}^{2}+\cdots+dx_{n}^{2}-dx_{n+1}^{2}.
\end{eqnarray*}
In fact, $\mathbb{R}^{n+1}_{1}$ is an $(n+1)$-dimensional Lorentz
manifold with index $1$. Denote by
\begin{eqnarray*}
\mathscr{H}^{n}(1)=\{(x_{1},x_{2},\cdots,x_{n+1})\in\mathbb{R}^{n+1}_{1}|x_{1}^{2}+x_{2}^{2}+\cdots+x_{n}^{2}-x_{n+1}^{2}=-1~\mathrm{and}~x_{n+1}>0\},
\end{eqnarray*}
which is exactly the hyperbolic plane of center $(0,0,\ldots,0)$
(i.e., the origin of $\mathbb{R}^{n+1}$) and radius $1$ in
$\mathbb{R}^{n+1}_{1}$. Clearly,
$\mathscr{H}^{n}(1)\subset\mathbb{R}^{n+1}_{1}$ is a spacelike
hypersurface  which is simply-connected Riemannian $n$-manifold with
constant negative curvature and is geodesically complete. This is
the reason why we call $\mathscr{H}^{n}(1)$ a hyperbolic plane. Let
$\Omega\subset\mathscr{H}^{n}(1)$ be a bounded domain with smooth
boundary $\partial\Omega$, and then consider the following Hessian
quotient equation with vanishing DBC
\begin{equation}\label{main-eq}
\left\{
\begin{aligned}
&
\frac{\sigma_{k}(\lambda(U[u]))}{\sigma_{l}(\lambda(U[u]))}=f(x,u,\nabla
u) \qquad &&\mathrm{in}~
\Omega\subset\mathscr{H}^{n}(1)\subset\mathbb{R}^{n+1}_{1}, \qquad
l+2\leq k\leq n,~l\geq0,
\\
&u=0 \qquad&&\mathrm{on}~\partial\Omega,
\end{aligned}
\right.
\end{equation}
where $u(x)$ is a function defined over $\Omega$, $U[u]:=\tau(\Delta
u)\cdot\mathbb{I}-\nabla^{2}u$ with $\tau\geq1$, and, with the abuse
of notations, $\nabla$, $\Delta$, $\nabla^{2}$ are the gradient, the
Laplace and the Hessian operators on $\mathscr{H}^{n}(1)$
respectively. For the Dirichlet problem (\ref{main-eq}), we can
prove:

\begin{theorem} \label{main-th}
Suppose that $k>l+1$, $u\in C^{4}(\Omega)\cap
C^{2}(\overline{\Omega})$ is a solution to the Hessian quotient
equation \eqref{main-eq} with $\lambda(U[u])\in \Gamma_{k}$, $f$ is
a positive smooth function. Then there exists a constant $\beta>0$
such that for $n\geq k\geq l+2$, we have
\begin{eqnarray*}
\sup\limits_{x\in\Omega}(-u)^{\beta}|\nabla^{2}u|(x)\leq C,
\end{eqnarray*}
where $C$ depends on $n$, $k$, $l$, $\tau$, $\sup_{\Omega}|u|$ and $
\sup_{\Omega}|\nabla u|$.
\end{theorem}

\begin{remark}
\rm{ (1) Obviously, by Theorem \ref{main-th}, one knows that if
$C^0$, $C^1$ estimates could be obtained for the $k$-convex
solutions to (\ref{main-eq}), then the interior $C^2$ estimates
follows directly.
\\
 (2) In the Dirichlet
problem (\ref{main-eq}), if the function $u$  was replaced by a
spacelike graphic function $\widetilde{u}$, the $(0,2)$-type tensor
$\widetilde{U}$ was replaced by the second fundamental form of the
spacelike graphic hypersurface determined by $\widetilde{u}$, and
the zero DBC was replaced by an affine function defined over
$\partial\Omega$, then (\ref{main-eq}) would become a prescribed
curvature problem (PCP for short) recently considered by the
corresponding author, Prof. J. Mao, and his
 collaborators. They successfully obtained the a priori estimates
 for $k$-admissible solutions to the PCP, and, together with the
 method of continuity, showed the existence and uniqueness of
 $2$-admissible solution to the PCP -- see \cite[Theorem 1.4]{ggm} for
 details. Of course, in \cite{ggm}, the $C^2$ interior estimates for the
 $k$-admissible solutions to the PCP are important to get the
 existence. This research experience, together with the continuous study of the
Dirichlet problem (\ref{hte-1}) and its more general version
(\ref{hte-2}) in different settings, is exactly our motivation of
considering the problem (\ref{main-eq}) in this paper.
 \\
 (3) Inspired by the work \cite{lrw}, it should be interesting to
 know whether the Pogorelov type estimates in Theorem \ref{main-th}
 is still valid for $k=l+1$. \\
 (4) Our Theorem \ref{main-th} here and \cite[Theorem 1.1]{ctx}
 somehow reveal the reasonability of considering \textbf{Problem 1}
and
 show the hope
 of possibly solving it. \\
 (5) Inspired by Theorem \ref{main-th} here, it is natural and feasible to try to improve the
 existing
 Pogorelov type estimates for solutions to some prescribed elliptic
 PDEs (with DBC) defined over bounded domains (with boundary) in the Euclidean space
 or Riemannian manifolds to our setting -- investigating elliptic
 PDEs of the same type defined over bounded domains (with boundary)
 in the spacelike hypersurface
 $\mathscr{H}^{n}(1)\subset\mathbb{R}^{n+1}_{1}$. We have already obtained an interesting result and now we prefer to leave this attempt
 to readers who are interested in this topic. \\
  (6) In fact, one can consider the problem (\ref{main-eq}) defined
  over more general bounded domains (with smooth boundary) in Lorentz-Minkowski space
  $\mathbb{R}^{n+1}_{1}$, and similar conclusion can be obtained --
  see Remark \ref{rm3-1} for details. \\
  (7) Since in our previous works \cite{GLM,ggm}, the PCPs therein were considered for spacelike graphic functions defined over bounded domains
  (with boundary) in the hyperbolic plane $\mathscr{H}^{n}(1)\subset\mathbb{R}^{n+1}_{1}$, in order to embody the continuity of our study on geometric elliptic
   PDEs in $\mathbb{R}^{n+1}_{1}$, we insist on giving the
  Pogorelov-type estimates in Theorem \ref{main-th} only for the
  case $\Omega\subset\mathscr{H}^{n}(1)$.
 }
\end{remark}

The paper is organized as follows. In Section \ref{S2}, some useful
formulas (including the structure equations for spacelike
hypersurfaces in $\mathbb{R}^{n+1}_{1}$, some basic properties of
elementary symmetric functions, etc) will be listed. The proof of
Theorem \ref{main-th} will be shown in Section \ref{S3}.

\section{Some useful formulae} \label{S2}

Let $g_{\mathscr{H}^{n}(1)}$ be the Riemannian metric on
$\mathscr{H}^{n}(1)$ induced by the Lorentzian metric $
\langle\cdot,\cdot\rangle_{L}$ of $\mathbb{R}^{n+1}_{1}$. For a
$(s,r)$-tensor field $\alpha$ on $\mathscr{H}^{n}(1)$, its covariant
derivative $\nabla \alpha$ is a $(s,r+1)$-tensor field given by
\begin{eqnarray*}
\begin{split}
&\nabla\alpha(Y^{1},\cdot\cdot\cdot,Y^{s},X_{1},\cdot\cdot\cdot,X_{r},X)\\
=~&\nabla_{X}\alpha(Y^{1},\cdot\cdot\cdot,Y^{s},X_{1},\cdot\cdot\cdot,X_{r})\\
=~&X(\alpha(Y^{1},\cdot\cdot\cdot,Y^{s},X_{1},\cdot\cdot\cdot,X_{r}))-
\alpha(\nabla _{X}Y^{1},\cdot\cdot\cdot,Y^{s},X_{1},\cdot\cdot\cdot,X_{r})\\
&-\cdot\cdot\cdot-\alpha(Y^{1},\cdot\cdot\cdot,Y^{s},X_{1},\cdot\cdot\cdot,\nabla
_{X}X_{r}),
\end{split}
\end{eqnarray*}
and its components in local coordinates
 are denoted by
\begin{eqnarray*}
\alpha^{l_{1}\cdots l_{s}}_{k_{1}\cdots k_{r}, k_{r+1}},
\end{eqnarray*}
where $1\leq l_{i},k_{j}\leq n$ with $i=1,2,\cdots,s$ and
$j=1,2,\cdots,r+1$. Here the comma ``," in subscript of a given
tensor means doing covariant derivatives. Besides, we make an
agreement that, for simplicity, in the sequel the comma ``," in
subscripts will be omitted unless necessary. One can continue to
define the second covariant derivative of $\alpha$ as follows:
\begin{eqnarray*}
\nabla^{2}\alpha(Y^{1},\cdots,Y^{s},X_{1},\cdots,X_{r},X,Y)=
(\nabla_{Y}(\nabla\alpha))(Y^{1},\cdots,Y^{s},X_{1},\cdots,X_{r},X).
\end{eqnarray*}
and then its components in local coordinates
 are denoted by
\begin{eqnarray*}
\alpha^{l_{1}\cdots l_{s}}_{k_{1}\cdots k_{r}, k_{r+1}k_{r+2}},
\end{eqnarray*}
where $1\leq l_{i},k_{j}\leq n$ with $i=1,2,\cdots,s$ and
$j=1,2,\cdots,r+2$. Similarly, the higher order covariant
derivatives of $\alpha$ are given as follows
\begin{eqnarray*}
\nabla^{3}\alpha=\nabla(\nabla^{2}\alpha),
~\nabla^{4}\alpha=\nabla(\nabla^{3}\alpha),\cdots,
\end{eqnarray*}
and so on.

On $\mathscr{H}^{n}(1)$, for any tangent vector fields $X,Y,Z$, the
Riemannian curvature $(1,3)$-tensor $R$ w.r.t.
$g_{\mathscr{H}^{n}(1)}$ is defined as
\begin{eqnarray*}
R(X,Y)Z=-\nabla_{X}\nabla_{Y}Z+\nabla_{Y}\nabla_{X}Z+\nabla_{[X,Y]}Z,
\end{eqnarray*}
where,  as usual, $[\cdot,\cdot]$ denotes the Lie bracket.
 In a local coordinate chart $\{\xi^{i}\}_{i=1}^{n}$ of
$\mathscr{H}^{n}(1)$, the component of the curvature tensor $R$ is
defined by
\begin{eqnarray*}
R\left(\frac{\partial}{\partial\xi^{i}},\frac{\partial}{\partial
\xi^{j}}\right)\frac{\partial}{\partial
\xi^{k}}=R_{kij}^{l}\frac{\partial}{\partial\xi^{l}},
\end{eqnarray*}
where $R_{ijkl}=\sigma_{jm}R_{ikl}^{m}$ and
$\sigma_{lm}:=g_{\mathscr{H}^{n}(1)}\left(\frac{\partial}{\partial
\xi^{l}},\frac{\partial}{\partial \xi^{m}}\right)$. Then, we have
the standard commutation formulas (i.e., Ricci identities)
\begin{eqnarray}\label{Ricci-id}
\qquad \alpha^{l_{1}\cdots l_{s}}_{k_{1}\cdots k_{r}, ji}
-\alpha^{l_{1}\cdots l_{s}}_{k_{1}\cdots k_{r},ij}
=-\sum\limits_{a=1}^{r}R_{ijk_{a}}^{m}\alpha^{l_{1}\cdots
l_{s}}_{k_{1}\cdots k_{a-1}mk_{a+1}\cdots k_{r}}
+\sum\limits_{b=1}^{s}R_{ijm}^{l_{b}}\alpha^{l_{1}\cdots
l_{b-1}ml_{b+1}\cdots l_{s}}_{k_{1}\cdots k_{r}}.
\end{eqnarray}
Denote by $X$ the position vector field of the spacelike
hypersurface $\mathscr{H}^{n}(1)\subset\mathbb{R}^{n+1}_{1}$.
Clearly, for any $x\in \mathscr{H}^{n}(1)$, $X(x)$ is a one-to-one
correspondence w.r.t. $x$. Let $\nu$ be the future-directed timelike
unit normal vector field and $h_{ij}$ be coefficient components of
the second fundamental form of the hypersurface $\mathscr{H}^{n}(1)$
w.r.t. $\nu$, that is,
\begin{eqnarray}\label{sec-fund}
h_{ij}=-\left\langle X_{,ij},\nu\right\rangle_{L}.
\end{eqnarray}
Recall the following identities
\begin{eqnarray}\label{gauss-form}
X_{,ij}=h_{ij}\nu, \qquad {\rm(Gauss~~formula)}
\end{eqnarray}
\begin{eqnarray}\label{weingarten-form}
\nu_{,i}=h_{ij}X^{j}, \qquad {\rm(Weingarten~~formula)}
\end{eqnarray}
where $X^{j}=\sigma^{ij}X_{i}$. Moreover, $\mathscr{H}^{n}(1)$ has
constant sectional curvature\footnote{~This fact will be shown
clearly in Section \ref{S3}.} $-1$ w.r.t. $g_{\mathscr{H}^{n}(1)}$
and satisfies the following Gauss equation
\begin{eqnarray*}
R_{ijkl}=\overline{R}_{ijkl}-\left(h_{ik}h_{jl}-h_{il}h_{jk}\right),\qquad
\qquad 1\leq i,j,k,l\leq n,
\end{eqnarray*}
which together with the fact $\overline{R}=0$ implies
\begin{eqnarray}\label{sectional-cur}
R_{ijkl}=-\left(h_{ik}h_{jl}-h_{il}h_{jk}\right),\qquad \qquad 1\leq
i,j,k,l\leq n,
\end{eqnarray}
where $\overline{R}$ stands for the curvature tensor of
$\mathbb{R}^{n+1}_{1}$. For a brief introduction to the structure
equations of spacelike
 hypersurfaces in $\mathbb{R}^{n+1}_{1}$, one can also check some of
 the corresponding author's previous works, e.g.,
 \cite{gm-1,gm-2,GLM,ggm}. We refer readers to an interesting book \cite{ha}
 about systematical knowledge on
 submanifolds in pseudo-Riemannian geometry.

At the end of this section, we prefer to give some properties of
elementary symmetric functions. Let
$\lambda=(\lambda_{1},\cdot\cdot\cdot,\lambda_{n})\in
\mathbb{R}^{n}$, and then for $1 \leq k \leq n$, the $k$-th
elementary symmetric function of $\lambda$ can be defined as follows
\begin{eqnarray*} \label{sigmak}
\sigma_{k}(\lambda)=\sum\limits_{1\leq
i_{1}<i_{2}<\cdot\cdot\cdot<i_{k}\leq
n}\lambda_{i_{1}}\lambda_{i_{2}}\cdot\cdot\cdot\lambda_{i_{k}}.
\end{eqnarray*}
We also set $\sigma_{0}=1$ and $\sigma_{k}=0$ for $k>n$ or $k<0$.
Recall that the Garding's cone is defined by
\begin{eqnarray} \label{gard-co}
\Gamma_{k}=\{\lambda\in \mathbb{R}^{n}|\sigma_{i}(\lambda)>0,\forall
1\leq i\leq k\}.
\end{eqnarray}
Denote by $\sigma_{k-1}(\lambda|i)=\frac{\partial
\sigma_{k}}{\partial \lambda_{i}}$. By, e.g., \cite{cns2},
\cite[Lemma 2.2.19]{G}, \cite{hs}, \cite[Chapter XV]{L}, one has the
following facts:

\begin{lemma}\label{prop-1}
Let $\lambda=(\lambda_{1},\cdot\cdot\cdot,\lambda_{n})\in \mathbb{R}^{n}$ and $1\leq k \leq n$, then we have\\
$(1)~~\Gamma_{1}\supset\Gamma_{2}\supset\cdot\cdot\cdot\supset\Gamma_{n};$\\
$(2)~~\sigma_{k-1}(\lambda|i)>0 ~for~ \lambda\in\Gamma_{k} ~and~ 1\leq i\leq n;$\\
$(3)~~\sigma_{k}(\lambda)=\sigma_{k}(\lambda|i)+\lambda_{i}\sigma_{k-1}(\lambda|i)~for~ 1\leq i\leq n;$\\
$(4)~~\left[\frac{\sigma_{k}}{\sigma_{l}}\right]^{\frac{1}{k-l}}
are~concave~and~elliptic~in~ \Gamma_{k}
~for~0\leq l<k;$\\
$(5)~~If~\lambda_{1}\geq\lambda_{2}\geq\cdot\cdot\cdot\geq\lambda_{n},~then~\sigma_{k-1}(\lambda|1)
\leq\sigma_{k-1}(\lambda|2)\leq\cdot\cdot\cdot\leq\sigma_{k-1}(\lambda|n)~for~\lambda\in\Gamma_{k};$\\
$(6)~~\sum\limits_{i=1}^{n}\sigma_{k-1}(\lambda|i)=(n-k+1)\sigma_{k-1}(\lambda).$\\
\end{lemma}

The following Newton-MacLaurin inequality will be used as well (see,
e.g., \cite{mt1,t2}).

\begin{lemma}\label{NM-ieq}
Let $\lambda \in \mathbb{R}^{n}$. For $0\leq l\leq k\leq n,~~r>s\geq
0,~~k\geq r,~~l\geq s$, we have\vspace {0.1 cm}
\begin{eqnarray*}
k(n-l+1)\sigma_{l-1}(\lambda)\sigma_{k}(\lambda)\leq
l(n-k+1)\sigma_{l}(\lambda)\sigma_{k-1}
\end{eqnarray*}
and
\begin{eqnarray*}
\left[\frac{\sigma_{k}(\lambda)/C_{n}^{k}}{\sigma_{l}(\lambda)/C_{n}^{l}}\right]^{\frac{1}{k-l}}
\leq
\left[\frac{\sigma_{r}(\lambda)/C_{n}^{r}}{\sigma_{s}(\lambda)/C_{n}^{s}}\right]^{\frac{1}{r-s}},
\qquad for~~ \lambda\in\Gamma_{k}.
\end{eqnarray*}
\end{lemma}

For convenience, we introduce the following notations
\begin{eqnarray}\label{not-}
\begin{split}
\quad F(U)=&\left[\frac{\sigma_{k}(U)}{\sigma_{l}(U)}\right]^{\frac{1}{k-l}},\quad T(D^{2}u)=F(\tau\Delta uI-D^{2}u),\\
\quad F^{ij}=\frac{\partial F}{\partial U_{ij}}&,\quad
F^{ij,rs}=\frac{\partial^{2} F}{\partial U_{ij}\partial U_{rs}},
\quad T^{ii}=\tau\sum\limits_{j=1}^{n}F^{jj}-F^{ii}.
\end{split}
\end{eqnarray}
Thus,
\begin{eqnarray*}
F^{ii}=\frac{1}{k-l}\left(\frac{\sigma_{k}(U)}{\sigma_{l}(U)}\right)^{\frac{1}{k-l}-1}
\frac{\sigma_{k-1}(U|i)\sigma_{l}(U)-\sigma_{k}(U)\sigma_{l-1}(U|i)}{\sigma_{l}^{2}(U)}.
\end{eqnarray*}
\vspace {0.1 cm}

To handle the ellipticity of the equation \eqref{main-eq}, we need
the following important proposition and its proof is almost the same
as that of \cite[Proposition 2.2.3]{C}.
\begin{lemma}\label{Ellip-}
Let $\lambda(U[u])\in \Gamma_{k}$ and $0\leq l\leq k-1$. Then the
operator
\begin{eqnarray}
F(U(u))=\left(\frac{\sigma_{k}(\lambda(U))}{\sigma_{l}(\lambda(U))}\right)^{\frac{1}{k-l}}
\end{eqnarray}
is elliptic and concave with respect to $U(u)$. Moreover we have
\begin{eqnarray}
\sum\limits_{i=1}^{n}F^{ii}\geq\left(\frac{C_{n}^{k}}{C_{n}^{l}}\right)^{\frac{1}{k-l}}.
\end{eqnarray}
\end{lemma}

We also need the following well-known result.

\begin{lemma}\label{F-Lamb}
Let $U$ be a diagonal matrix with $\lambda(U[u])\in \Gamma_{k}$,
$0\leq l\leq k-1$ and $k\geq3.$ Then
\begin{eqnarray}\label{FL-1}
-F^{1i,i1}(U)=\frac{F^{11}-F^{ii}}{U_{ii}-U_{11}}
\end{eqnarray}
for $i\geq2$. Moreover, if $U_{11}\geq U_{22}\geq\cdot\cdot\cdot\geq
U_{nn}$, we have
\begin{eqnarray}\label{FL-2}
F^{11}\leq F^{22}\leq \cdot\cdot\cdot\leq F^{nn}.
\end{eqnarray}
\end{lemma}

\begin{proof}
See \cite[Proposition 2.3]{ctx} for the proof of \eqref{FL-1}. The
proof of \eqref{FL-2} is almost the same with that of \cite[Lemma
2.2]{A} and we prefer to omit here.
\end{proof}

\section{Proof of Theorem \ref{main-th}} \label{S3}

In this section, we will use an idea (which is similar to that in
\cite{ctx,cj}) to give the proof of Theorem \ref{main-th}.

Assume that $u\in C^{4}(\Omega)\cap C^{2}(\overline{\Omega})$ is a
solution of the problem \eqref{main-eq} with $\lambda(U)\in
\Gamma_{k}$. Without loss of generality, we assume $u<0$ in $\Omega$
by the maximum principle.

Assume that a point on $\mathscr{H}^{n}(1)$ is described by local
coordinates $\xi^{1},\ldots,\xi^{n}$, that is,
$x=x(\xi^{1},\ldots,\xi^{n})$. For convenience, let
$\partial_{i}:=\partial/\partial\xi^{i}$ be the corresponding
coordinate fields on $\mathscr{H}^{n}(1)$ and then
$\sigma_{ij}=g_{\mathscr{H}^{n}(1)}(\partial_i,\partial_j)$ be the
Riemannian metric on $\mathscr{H}^{n}(1)$. For spacelike graphic
hypersurfaces
$\mathcal{G}:=\{(x,\widetilde{u}(x))|x\in\mathscr{H}^{n}(1)\}$ in
$\mathbb{R}^{n+1}_{1}$, one knows that its induced metric from the
Lorentzian metric $\langle\cdot,\cdot\rangle_{L}$ of
$\mathbb{R}^{n+1}_{1}$ has the form
$g:=u^{2}g_{\mathscr{H}^{n}(1)}-dr^{2}$. By \cite[Lemma 3.1]{gm-1}
(see also \cite{gm-2,GLM,ggm}), it is easy to know that for
$\mathcal{G}$, the tangent vectors, the future-directed timelike
unit normal vector and the second fundamental form are given by
\begin{eqnarray*}
X_{i}=\partial_{i}+\widetilde{u}_i\partial_{r}, \qquad
i=1,2,\cdots,n,
\end{eqnarray*}
\begin{eqnarray*}
\nu=\frac{1}{v}\left(\partial_r+\frac{1}{\widetilde{u}^2}\widetilde{u}^j\partial_j\right),
\end{eqnarray*}
with $\widetilde{u}^{j}:=\sigma^{ij}\widetilde{u}_{i}$,
$v:=\sqrt{1-\widetilde{u}^{-2}|\nabla \widetilde{u}|^2}$, and
\begin{eqnarray*}
h_{ij}=-\frac{1}{v}\left(\frac{2}{\widetilde{u}}{\widetilde{u}_i
\widetilde{u}_j }-\widetilde{u}_{ij} -\widetilde{u}
\sigma_{ij}\right),
\end{eqnarray*}
Since $\mathscr{H}^{n}(1)$ can be seen as the special case of
$\mathcal{G}$ with $\widetilde{u}\equiv1$, one has
\begin{eqnarray*}
\nu=\partial_{r},\qquad h_{ij}=\sigma_{ij},
\end{eqnarray*}
 which implies further the Gauss equation (\ref{sectional-cur}) has
 the following form
\begin{eqnarray*}
R_{ijkl}=-(\sigma_{ik}\sigma_{jl}-\sigma_{il}\sigma_{jk}),\qquad
\qquad 1\leq i,j,k,l\leq n
\end{eqnarray*}
for the spacelike hypersurface $\mathscr{H}^{n}(1)$. Then by Schur's
theorem (see, e.g., \cite[Chapter 4]{cl}),  one knows:\footnote{~We
prefer to mention readers that the computation
$R_{ijkl}=\sigma_{jm}R_{ikl}^{m}$ (for components of curvature
tensor) used in Section \ref{S2} has an opposite sign with the
corresponding one used in \cite[Chapter 4]{cl}. However, there is no
essential difference between two settings -- when using them to
calculate sectional (or Ricci, scalar) curvatures, they coincide
with each other. The reason why one meets two settings for
$R_{ijkl}$ is that opposite orientations have been chosen for the
unit normal vector when computing components of the second
fundamental form.}
\begin{itemize}
\item $\mathscr{H}^{n}(1)$ has constant sectional curvature $-1$.
\end{itemize}
Obviously, when $\mathcal{G}$ degenerates into $\mathscr{H}^{n}(1)$,
its metric $g$ becomes $g_{\mathscr{H}^{n}(1)}$ directly. Now, we
would like to give more information about the induced metric
$g_{\mathscr{H}^{n}(1)}$ of
$\mathscr{H}^{n}(1)\subset\mathbb{R}^{n+1}_{1}$ with components
$\sigma_{ij}=g_{\mathscr{H}^{n}(1)}(\partial_i,\partial_j)$ for any
$1\leq i,j,k,l\leq n$. As we know, for any point on
$\mathscr{H}^{n}(1)\subset\mathbb{R}^{n+1}_{1}$, its global
Lorentz-Minkowski coordinates $(x^{1},x^{2},\cdots,x^{n+1})$ can be
reparameterized as follows
 \begin{equation} \label{def-33}
 \left\{
\begin{aligned}
&x^{1}=\cos\xi^{1}\cos\xi^{2}\dots\cos\xi^{n-1}\sinh\xi^{n}\\
&x^{2}=\cos\xi^{1}\cos\xi^{2}\dots\sin\xi^{n-1}\sinh\xi^{n}\\
&\quad\dots~\dots\\
&x^{n-1}=\cos\xi^{1}\sin\xi^{2}\sinh\xi^{n}\\
&x^{n}=\sin\xi^{1}\sinh\xi^{n}\\
&x^{n+1}=\cosh\xi^{n},
\end{aligned}
\right.
\end{equation}
which implies
\begin{eqnarray*}
&&g_{\mathscr{H}^{n}(1)}=\sinh^{2}\xi^{n}(d\xi^{1})^{2}+\sinh^{2}\xi^{n}\cos^{2}\xi^{1}(d\xi^{2})^{2}
+\cdots\\
&&\qquad\qquad\qquad+\sinh^{2}\xi^{n}\cos^{2}\xi^{1}\dots\cos^{2}\xi^{n-2}(d\xi^{n-1})^{2}+(d\xi^{n})^{2},
\end{eqnarray*}
and then
\begin{eqnarray*}
&&\sigma_{ij}=\delta_{ij}\sinh^{2}\xi^{n}\cos^{2}\xi^{1}\dots\cos^{2}\xi^{i-1},
\\
&&\sigma_{11}=\sinh^{2}\xi^{n},\quad\sigma_{nn}=1,\quad\sigma_{1i}=\sigma_{i1}=\sigma_{ni}=\sigma_{in}=0,
\end{eqnarray*}
for any $2\leq i,j\leq n-1$, with the inverse
\begin{eqnarray*}
&&\sigma^{ij}=\delta^{ij}\sinh^{-2}\xi^{n}\cos^{-2}\xi^{1}\dots\cos^{-2}\xi^{i-1},
\\
&&\sigma^{11}=\sinh^{-2}\xi^{n},\quad\sigma^{nn}=1,\quad\sigma^{1i}=\sigma^{i1}=\sigma^{ni}=\sigma^{in}=0.
\end{eqnarray*}
Choose a fixed point $q\in\mathscr{H}^{n}(1)$ which is outside the
bounded domain $\Omega\subset\mathscr{H}^{n}(1)$, and define a
function $\rho$ on $\Omega$ as follows
\begin{eqnarray} \label{nf}
\rho:=\rho(q,x), \quad \forall x\in\Omega,
\end{eqnarray}
where $\rho(q,x)$ measures the Riemannian distance between $q$ and
$x$. Let $\gamma(t)$ be the minimizing geodesic connecting $q$ and
$x$, with $\gamma(0)=q$, $\gamma(\rho)=x$. Let
$\{\frac{d\gamma(t)}{dt}|_x, Y_{1}, \dots, Y_{n-1}\}$ be an
orthonormal basis of the tangent space
$T_{\gamma(\rho)}\mathscr{H}^{n}(1)$. Parallel transport vectors
$Y_{1}, \dots, Y_{n-1}$ along the geodesic $\gamma(t)$ yields
orthogonal vector fields $\widetilde{Y}_{i}$, $i=1,2,\cdots,n-1$.
Clearly, each $\widetilde{Y}_{i}$ is the Jaccobi field along
$\gamma(t)$ satisfying $\widetilde{Y}_{i}(\rho)=Y_{i}$. Since
$\mathscr{H}^{n}(1)$ has constant sectional curvature $-1$, one has
\begin{eqnarray*}
\widetilde{Y}_{i}(t)=f(t)Y_{i}(t)
\end{eqnarray*}
with $f(t)=\frac{1}{\sinh\rho}\sinh t $, and then
\begin{eqnarray*}
\begin{split}
\Delta\rho&=\sum\limits_{i=1}^{n-1}\mathrm{Hess}(\rho)(Y_{i},Y_{i})\\
&=\sum\limits_{i=1}^{n-1}\int_{0}^{\rho}\left(|\nabla_{\frac{\partial}{\partial t}}\widetilde{Y}_{i}|^{2}+
\langle \widetilde{Y}_{i},\nabla_{\frac{\partial}{\partial t}}\nabla_{\frac{\partial}{\partial t}}\widetilde{Y}_{i}\rangle_{g_{\mathscr{H}^{n}(1)}}\right)dt\\
&=(n-1)\int_{0}^{\rho}\left(\left|\frac{df(t)}{dt}\right|^{2}+f^{2}(t)\right)dt\\
&=(n-1)\int_{0}^{\rho}\frac{1}{\sinh^{2}\rho}(\sinh^{2}t+\cosh^{2}t)dt\\
&=(n-1)\coth\rho,
\end{split}
\end{eqnarray*}
 where $\mathrm{Hess}$ is the Hessian operator on
 $\mathscr{H}^{n}(1)$, and
 $\langle\cdot,\cdot\rangle_{g_{\mathscr{H}^{n}(1)}}$ is the inner
 product w.r.t. the metric $g_{\mathscr{H}^{n}(1)}$. Since $\Omega$
 is bounded and complete, from the definition (\ref{nf}), it is easy
 to know that $\rho=\rho(q,x)$ has infimum $c^{-}>0$ and supremum
 $c^{+}$ simultaneously. Therefore, we have
 \begin{eqnarray} \label{nd-2}
 (n-1)\coth(c^{+})\leq\Delta\rho(q,x)\leq(n-1)\coth(c^{-}), \qquad
 \forall x\in\Omega.
 \end{eqnarray}
On $\Omega$, consider the following test function
\begin{eqnarray*}
\widetilde{P}(x)=\beta
\log(-u)+\log\lambda_{\max}(x)+\frac{a}{2}|\nabla u|^{2}+A\cdot
\rho(q,x),
\end{eqnarray*}
where $\rho(q,x)$ is defined as (\ref{nf}), $\lambda_{\max}(x)$ is
the biggest eigenvalue of the Hessian matrix $u_{ij}$, and $\beta$,
$a$, $A$ are positive constants which will be determined later.
Suppose that $\widetilde{P}$ attains its maximum value in $\Omega$
at $x_{0}$. Choosing a local orthonormal frame field at $x_{0}$ such
that $g_{ij}(x_{0})=\delta_{ij}(x_{0})$. This can always be assured
-- e.g., choosing local coordinates
$\xi^{1},\xi^{2},\cdots,\xi^{n-1},\sinh^{-1}\xi^{n}$ around $x_{0}$,
where $\xi^{i}$, $i=1,2,\cdots,n$ are determined by (\ref{def-33}).
Rotating further the coordinate axes, we can diagonal the matrix
$\nabla^{2}u=(u_{ij})$,
\begin{eqnarray*}
u_{ij}(x_{0})=u_{ii}(x_{0})\delta_{ij},\quad u_{11}(x_{0})\geq
u_{22}(x_{0})\geq\cdot\cdot\cdot \geq u_{nn}(x_{0}),
\end{eqnarray*}
and then
\begin{eqnarray*}
U_{11}(x_{0})\leq U_{22}(x_{0})\leq\cdots\leq U_{nn}(x_{0}).
\end{eqnarray*}
So, by \eqref{FL-2} we can obtain
\begin{eqnarray*}\label{ineq-1}
F^{11}(x_{0})\geq F^{22}(x_{0})\geq\cdots\geq F^{nn}(x_{0})>0,
\end{eqnarray*}
\begin{eqnarray*}\label{ineq-2}
0<T^{11}(x_{0})\leq T^{22}(x_{0})\leq\cdots\leq T^{nn}(x_{0}).
\end{eqnarray*}

On $\Omega$, define a new function
\begin{eqnarray*}
P(x)=\beta \log(-u)+\log u_{11}(x)+\frac{a}{2}|\nabla u|^{2}+A\cdot
\rho(q,x),
\end{eqnarray*}
which also attains a local maximum at $x_{0}$. Differentiating $P$
at $x_{0}$ once yields
\begin{eqnarray*}\label{dif-1}
\frac{\beta u_{i}}{u}+\frac{u_{11i}}{u_{11}}+au_{ii}u_{i}+A\cdot
\rho_{i}(q,x)=0,
\end{eqnarray*}
and differentiating $P$ at $x_{0}$ twice results in
\begin{eqnarray*}\label{dif-2}
\frac{\beta u_{ii}}{u}-\frac{\beta
u_{i}^{2}}{u^{2}}+\frac{u_{11ii}}{u_{11}}
-\frac{u_{11i}^{2}}{u_{11}^{2}}+a\sum\limits_{p=1}^{n}u_{p}u_{pii}+au_{ii}^{2}
+A\cdot \rho_{ii}(q,x)\leq0.
\end{eqnarray*}
Thus, at $x_{0}$,
\begin{eqnarray}\label{main-ieq}
\begin{split}
0\geq &~T^{ii}P_{ii}\\
= &~\frac{T^{ii}\beta u_{ii}}{u}-\frac{\beta
T^{ii}u_{i}^{2}}{u^{2}}+\frac{T^{ii}u_{11ii}}{u_{11}}
-\frac{T^{ii}u_{11i}^{2}}{u_{11}^{2}}\\
&+a\sum\limits_{p=1}^{n}u_{p}T^{ii}u_{pii}+aT^{ii}u_{ii}^{2}
+AT^{ii}\rho_{ii}\\
\geq &~\frac{T^{ii}\beta u_{ii}}{u}-\frac{\beta
T^{ii}u_{i}^{2}}{u^{2}}+\frac{T^{ii}u_{11ii}}{u_{11}}
-\frac{T^{ii}u_{11i}^{2}}{u_{11}^{2}}\\
&+a\sum\limits_{p=1}^{n}u_{p}T^{ii}u_{pii}+aT^{ii}u_{ii}^{2}
+AT^{11}\Delta\rho\\
\geq &~\frac{T^{ii}\beta u_{ii}}{u}-\frac{\beta
T^{ii}u_{i}^{2}}{u^{2}}+\frac{T^{ii}u_{11ii}}{u_{11}}
-\frac{T^{ii}u_{11i}^{2}}{u_{11}^{2}}\\
&+a\sum\limits_{p=1}^{n}u_{p}T^{ii}u_{pii}+aT^{ii}u_{ii}^{2}
+AT^{11} (n-1)\coth(c^{+}),
\end{split}
\end{eqnarray}
 where the Einstein summation convention has been used --  repeated superscripts and subscripts should be
made summation. The second inequality in (\ref{main-ieq}) holds
since $\Delta\rho$ is globally defined and is independent of the
choice of local coordinates. The last inequality in (\ref{main-ieq})
holds because of the estimate (\ref{nd-2}).
 If one goes back to the proof of \cite[Theorem
1.1]{ctx}, he or she would find that for calculations done at a
point, the above inequality has \emph{nearly same form} with (3.5)
of \cite{ctx} except the last term in the RHS -- the last term in
the RHS of (\ref{main-ieq}) is $AT^{11} (n-1)\coth(c^{+})$ while the
one in (3.5) of \cite{ctx} is $A\sum_{i=1}^{n}T^{ii}$. However, the
main role of $A$ in the evaluation of each term in (3.5) of
\cite{ctx} is to absorb those \emph{negative} terms appearing.
Since, in our setting here, $T^{ii}(x_0)$ can be controlled at $x_0$
and $(n-1)\coth(c^{+})$ is a constant (which will not break the role
of $A$), with the help of properties of $\sigma_{k}$-operator and
structure equations listed in Section \ref{S2}, our Pogorelov type
estimates in Theorem \ref{main-th} follows by using a similar
argument (with necessary modifications\footnote{~Of course, our
positive constants $\beta$, $a$ and $A$ here might be different from
those in the function $\widetilde{P}$ constructed in the proof of
\cite[Theorem 1.1]{ctx}, since constants $n$ and $c^{+}$ appear.})
to the rest part of the proof of \cite[Theorem 1.1]{ctx}.

Of course, in the evaluating process of terms in the RHS of
(\ref{main-ieq}), the identity
\begin{eqnarray*}
\begin{split}
u_{i11}&=u_{11i}+u_{j}R^{j}_{1i1}\\
&=u_{11i}+u^{m}(\delta_{11}\delta_{im}-\delta_{1m}\delta_{i1})
\end{split}
\end{eqnarray*}
with $u^{m}=u_{j}\delta^{jm}$, which was obtained by
\eqref{Ricci-id} and \eqref{sectional-cur}, would be used. However,
the curvature term $R^{j}_{1i1}$ will not bring us any trouble,
since its boundedness over the bounded domain $\Omega$ is enough
(let alone it is a constant in our setting).

\vspace{3mm}

\begin{remark} \label{rm3-1}
\rm{ (1) Let $\mathcal{M}$ be an immersed complete spacelike
hypersurface in $\mathbb{R}^{n+1}_{1}$ and assume that
$\Omega\subset\mathcal{M}$ is a bounded domain (with smooth boundary
$\partial\Omega$). If furthermore $\mathcal{M}$ is simply connected
and satisfies
 \begin{eqnarray*}
 \mathrm{Ric}(\mathcal{M})\geq -(n-1)K \qquad\mathrm{and}\qquad
 \mathrm{Sec}(\mathcal{M})\leq0,
 \end{eqnarray*}
where $K\geq0$ is a nonnegative constant, and $
\mathrm{Ric}(\mathcal{M})$, $\mathrm{Sec}(\mathcal{M})$ denote the
Ricci curvature and the sectional curvature of $\mathcal{M}$
respectively, then by Hessian comparison theorem and Laplace
comparison theorem (see, e.g., \cite{ys}), one has
\begin{eqnarray*}
 \frac{n-1}{\rho} \leq \Delta\rho \leq
 \frac{n-1}{\rho}\sqrt{K}\rho\coth(\sqrt{K}\rho)\leq\frac{n-1}{\rho}\left(1+\sqrt{K}\rho\right).
\end{eqnarray*}
 Together with the fact $0<c^{-}\leq\rho\leq c^{+}$, it is easy to
 know that
  \begin{eqnarray*}
0<\frac{n-1}{c^{+}}\leq\Delta\rho\leq \sqrt{K}+\frac{n-1}{c^{-}}.
  \end{eqnarray*}
Inspired by this observation and the construction of auxiliary
functions in our proof of Theorem \ref{main-th}, it is not hard to
know that our Pogorelov type estimates in Theorem \ref{main-th}
would still be valid if the domain $\Omega\subset\mathscr{H}^{n}(1)$
in the problem (\ref{main-eq}) was replaced by
$\Omega\subset\mathcal{M}$ mentioned as above. Of course, in this
setting, $\mathcal{M}$ must be noncompact since $\mathcal{M}$ is
simply connected and $\mathrm{Sec}(\mathcal{M})\leq0$ (using
Cartan-Hadamard theorem directly, see, e.g., \cite[Chapter 5]{cl}).
\\
 (2) One might see that Dai's main conclusion in \cite{qd} can be improved
to bounded domains (with smooth boundary) of simply connected
complete Riemannian manifolds whose sectional curvature is
non-positive and whose Ricci curvature is bounded from below by
non-positive constant.  What about bounded domains on complete
manifolds with positive curvature? If one checks our proof here
carefully, one might see that the key point is how to find a
strictly positive lower bound for $\Delta\rho$. Based on this
reason, maybe complete manifolds with \emph{suitable pinching
assumption for positive curvature} could be expected to get similar
conclusion as well. Readers who have interest can try this
improvement. So far, we do not have this interest to finish this,
since what we are caring about is the Lorentz-Minkowski situation.
 }
\end{remark}

\section*{Acknowledgments}
This work is partially supported by the NSF of China (Grant Nos.
11801496 and 11926352), the Fok Ying-Tung Education Foundation
(China) and  Hubei Key Laboratory of Applied Mathematics (Hubei
University). The authors would like to thank Dr. Ya Gao and Ms. Qing
Dai for useful discussions during the preparation of this paper.

\end{document}